\documentclass{amsart}
\usepackage{amsmath}
  \usepackage{paralist}
  \usepackage{graphics} 
  \usepackage{epsfig} 
\usepackage{graphicx}  \usepackage{epstopdf}
 \usepackage[colorlinks=true]{hyperref}
\hypersetup{urlcolor=blue, citecolor=red}

\newtheorem{thm}{Theorem}[section]
\newtheorem{prop}[thm]{Proposition}
\newtheorem{cor}[thm]{Corollary}
\newtheorem{lem}[thm]{Lemma}

\newtheorem{remark}[thm]{Remark}
\theoremstyle{defi}

\title[Maximal estimates for the Kramers-Fokker-Planck-Magnetic operator ] 
      {Maximal estimates for the Kramers-Fokker-Planck operator with electromagnetic field}

\author[Bernard Helffer $\&$ Zeinab Karaki]{Bernard Helffer$^{(1)}$ $\&$ Zeinab Karaki$^{(2)}$}


 \email{bernard.helffer@univ-nantes.fr; zeinab.karaki@outlook.com}

\begin{document}
\maketitle

{\footnotesize 
\centerline{ $(1)$  Universit\'e de Nantes}
   \centerline{Laboratoire de Math\'ematiques Jean Leray}
   \centerline{ 2, rue de la Houssini\`ere}
    \centerline{BP 92208 F-44322 Nantes Cedex 3, France}}
    {\footnotesize 
  \centerline{  $(2)$    Universit\'e de Perpignan  Via Domitia}
   \centerline{Laboratoire de Math\'ematiques et Physique}
   \centerline{ 52 AVE Paul Alduy}
    \centerline{66860 Perpignan Cedex 9, France}
} 

\medskip

\bigskip


 \begin{abstract} 
In continuation of a former work  by the first author with  F. Nier (2009) and of a more recent work by the second author on the torus (2019), we consider the Kramers-Fokker-Planck operator (KFP) with an external electromagnetic field on $\mathbb R^d$. We show a maximal type estimate on this operator using a nilpotent approach for vector field polynomial operators and induced representations of a nilpotent graded Lie algebra. This estimate leads to an  optimal characterization of the domain of the closure of the KFP operator and a criterion for the compactness of the resolvent.
\end{abstract}

\section{Introduction and main results}
\subsection{Introduction}
The Fokker-Planck equation was introduced by Fokker   and Planck at the beginning of the twentieth century, to describe the evolution of the density of particles under the Brownian motion. In recent years, global hypoelliptic estimates have led to new results
motivated by  applications to the kinetic theory of gases. In
this direction many authors have shown maximal estimates
to deduce the compactness of the resolvent of the Fokker-Planck operator and to have
resolvent estimates in order  to address the issue of return to
the equilibrium. F. H\'erau and F. Nier in \cite{herau2004isotropic} have highlighted the
links between the Fokker-Planck operator with a confining potential and
the associated  Witten Laplacian. Later this work has been extended  in the book of B. Helffer
and F. Nier \cite{HeNi}, and we refer more specifically to their Chapter 9 for a proof of
the maximal estimate.

In this article, we continue the study of the model case of the operator
of Fokker-Planck with an external magnetic field $B_e$, which was initiated in the case of the torus $\mathbb T^d$ ($d=2,3$) in 
\cite{ZK,zk2}, by considering $\mathbb R^d$  and reintroducing an electric potential as in \cite{HeNi}.  In this context  we establish a maximal-type estimate for this model, giving a
characterization of the domain of its closed extension and giving sufficient conditions for the compactness of the resolvent. 
\subsection{Statement of the result}
For $d=2$ or $3$, we consider, for a given  external electromagnetic field  $B_e$ defined on 
$\mathbb R^d$ with value in $\mathbb{R}^{d(d-1)/2}$ and a real valued  electric potentiel $V$ defined on
$\mathbb R^d$,  the associated Kramers-Fokker-Planck operator $K$ (in short KFP) defined by:
\begin{align}\label{def_K}
 K=v\cdot \nabla_x -\nabla_x V\cdot \nabla_v-(v\wedge B_e )\cdot \nabla_v -\Delta_v + v^2/4 -d/2, 
\end{align}
where $v\in\mathbb{R}^{d}$ represents the velocity, $x \in \mathbb{R}^d$ represents the space variable, and the notation $(v \wedge B_e ) \cdot \nabla_v$ means:
\begin{equation*}
(v\wedge B_e)\cdot \nabla_{v}=\begin{cases}
b(x)\,(v_1 \partial_{v_{2}} -v_2 \partial_{v_{1}}) \quad \quad &\text{ if } d=2 \\~\\
b_1(x)(v_2 \partial_{v_{3}} -v_3 \partial_{v_{2}})+b_2(x) (v_3 \partial_{v_{1}} -v_1 \partial_{v_{3}})\\
+ b_3(x)(v_1 \partial_{v_{2}} -v_2 \partial_{v_{1}})  \quad \qquad
&\text{ if } d=3.
\end{cases}
\end{equation*}
 The operator $K$ is initially considered as an unbounded operator on the Hilbert space $L^2 (\mathbb{R}^{d}\times \mathbb{R}^{d})$ whose domain is 
  $D(K)=C_0^\infty(\mathbb R^d\times \mathbb{R}^d)$.\\
   We then denote by:
\begin{enumerate}
\item[$\bullet$] $K_{\min}$ the minimal extension of $K$  where $D(K_{min})$ is the closure of $D(K)$ with respect to the graph norm;
 \item[$ \bullet $] $K_{\max}$  the maximal extension of $K$ where $D(K_{\max})$ is given by:
$$ D(K_{\max}) =\{u\in L^2 (\mathbb{R}^{d}\times \mathbb{R}^{d}) \,/\, Ku\in L^2 (\mathbb{T}^{d}\times \mathbb{R}^{d})\}.$$
\end{enumerate}
We will use the notation $\bf{K}$ for the operator $K_{min}$ or  ${\bf K}_{B_e,V}$ if we want to mention the reference to $B_e$ and $V$.

 The existence of a strongly
continuous semi-group associated to operator $\bf K$ is shown in  \cite{ZK} when the magnetic field is regular and $V=0$. We will  improve this result by considering a much lower regularity. In order to obtain the maximal accretivity, we are led to substitute
the hypoellipticity argument by a regularity argument for the operators with coefficients in $ L^\infty_{loc} $, which will be combined with more classical results of Rothschild-Stein in \cite{rothschild1979criterion} for H\"ormander operators of type 2 (see \cite{hormander1967hypoelliptic} for more details of this subject). Our first result is:
\begin{thm}\label{prop 2}
If $B_{e}\in L^{\infty}_{loc}(\mathbb{R}^{d},\mathbb{R}^{d(d-1)/2})$ and $ V\in W^{1,\infty}_{loc}( \mathbb  R^d )$, then ${\bf K}_{B_e,V}$  is maximally accretive.
\end{thm}
The theorem implies that the domain of the operator ${\bf K}=K_{min}$ has the following property:
\begin{align}
\label{eq:1} 
D({\bf K})= D(K_{\max})\,.
\end{align}
We are next  interested in specifying the domain of the operator $ \bf K
  $ introduced in \eqref{eq:1}. For this goal, we will establish  a  maximal estimate  for $ \bf K $, using techniques which were  developed initially for the study of hypoellipticity of invariant operators on nilpotent groups and the proof of the Rockland conjecture. Before we state our main result, we introduce the following functional spaces:
\begin{itemize}
\item $B^{2}(\mathbb R^d)$ (or $B^2_v$ to indicate the name of the variables) denotes the space:
 $$B^2(\mathbb R^d) :=\{u\in L^{2}(\mathbb{R}^{d})\,/\, \forall (\alpha, \beta )\in \mathbb{N}^{2d},\, \vert \alpha \vert +\vert \beta \vert \leq 2 \,,\, v^{\alpha}\,\partial^{\beta}_{v}\,u\in L^{2}(\mathbb{R}^{d}) \}, $$
 which is equipped with its natural Hilbertian norm.
\item $ \tilde{B}^2 (\mathbb{R}^d \times \mathbb{R}^d) $ is the space $ L_{x}^{2} \widehat \otimes B^{2}_{v} $ (in $ L^2 (\mathbb{R}^d \times \mathbb{R}^d) $ identified with $ L_{x}^{2} \widehat \otimes L^{2}_{v} $) with its natural Hilbert norm.
 \end{itemize}
We can now state the second  theorem of this article:
\begin{thm}\label{hypoelli0}
Let $ d=2 \text{ or } 3$. We assume that $ B_e \in C^1(\mathbb{R}^{d},\mathbb R^{d(d-1)/2}) \cap L^\infty$ and that there exist positive constants $C$, $\rho_0 >\frac 13$  and $\gamma_0<\frac{1}{3}$ such that 
 
 \begin{equation} 
 \vert \nabla_x B_e(x)\vert \leq C <\nabla V (x)>^{\gamma_0} \,,
 \end{equation} 
\begin{equation}\label{hyp2/3}
|D_x^\alpha V (x)| \leq C <\nabla V (x)>^{1-\rho_0}\,,\, \forall \alpha \mbox{ s.t. } |\alpha|=2\,,
\end{equation}
where $$ <\nabla V (x)>=\sqrt{|\nabla V (x)|^2 +1}\,.
$$

 Then  there exists $C_1>0$  such that, for all $u \in C_0^\infty (\mathbb{R}^{d}\times \mathbb{R}^{d})$, the operator $K$ satisfies the following maximal estimate:
\begin{align}
||\, |\nabla V (x)|^\frac 23 \, u \, ||+\Vert (v\cdot \nabla_{x} -\nabla_x V\cdot \nabla_v -(v\wedge B_e )\cdot \nabla_v )u\Vert+\Vert u \Vert_{\tilde{B}^{2}} \leq C_1 (\,\Vert K u\Vert +\Vert u\Vert\, ) .
\label{hypomax}
\end{align}
\end{thm}
The proofs will combine the previous works of \cite{HeNi} (in the case $B_e=0$) and  \cite{zk2} (in the case $V=0$)  with in addition two differences: 
\begin{itemize}
\item $\mathbb T^d$ is replaced by $\mathbb R^d$\,.
\item The reference operator in the enveloping algebra of the nilpotent algebra is different.
\end{itemize}
Notice also  that, when $B_e=0$,  our assumptions are weaker than in the book \cite{HeNi} where the property that $|\nabla V (x)|$ tends to $+\infty$ as $|x| \rightarrow +\infty$ was used to construct the partition of unity.\\
Using the density of  $C_0^\infty(\mathbb{R}^d \times \mathbb{R}^d)$  in the domain of  $\bf K$, we obtain the following characterization of this domain:
\begin{cor}
\begin{equation}\label{cardom}
\begin{array}{ll}
 D({ \bf K} )=\{u\in \tilde B ^2(\mathbb{R}^{d}\times\mathbb{R}^{d}) \, /\,\\  \qquad\qquad     \left(v\cdot\nabla_{x} -\nabla V\cdot \nabla_v -(v\wedge B_e )\cdot \nabla_{v}\right) u \mbox{ and } 
|\nabla V (x)|^\frac 23 \, u  \in L^{2}(\mathbb{ R}^{d}\times\mathbb{R}^{d}) \}.
\end{array}
\end{equation}
\end{cor}
In particular this implies that  under the assumptions of Theorem \ref{hypoelli0} the operator ${ \bf K}_{B_e,V} $  has compact resolvent if and only if ${ \bf K}_{B_e=0,V} $ has the same property. This is in particular the case  (see \cite{HeNi}) when 
\begin{equation*}
|\nabla V (x) | \to + \infty \text{ when }  |x|\to +\infty\,,
\end{equation*}
as can also be seen directly from \eqref{cardom}. 
\begin{remark}For more results in the case without magnetic field we refer to \cite{HeNi} and  recent  results obtained  in 2018 by Wei-Xi Li \cite{WXL} and in 2019 by M. Ben Said \cite{BS} in connexion with a conjecture of Helffer-Nier relating the compact resolvent property for the (KFP)-operator with the same property for the Witten Laplacian on $(0)$-forms: $-\Delta_{x,v}+\frac 14  |\nabla \Phi|^2 - \frac 12 \Delta \Phi$, with $\Phi(x,v)= V(x) + \frac{v^2}{2}$. Its proof involves also nilpotent techniques. One can then naturally ask on results when $B_e(x)$ is unbounded. In particular, the existence of  (KFP)-magnetic bottles (i.e. the compact resolvent property for the (KFP)-operator) when $V=0$ is natural. Here we simply observe that Proposition~5.19 in \cite{HeNi} holds (with exactly the same proof) for ${\bf K}_{B_e,V}$ when $B_e$ and $V$ are $C^\infty$. Hence there are no (KFP)-magnetic bottles.
\end{remark}
\section{Maximal accretivity for the Kramers-Fokker-Planck operator with a weakly regular electromagnetic field}
 To prove Theorem \ref{prop 2}, we will show the Sobolev regularity associated to the following problem
$$
 K^*f=g \mbox{  with } f,g\in L^2_{loc}(\mathbb{R}^{2d})\,,$$ where $K^*$ is the formal adjoint of $ K $:
\begin{align}
 K^*=-v\cdot\nabla_x-\Delta_v+(v\wedge B_e+\nabla_x V)\cdot \nabla_v +v^2/4-d/2\,.\label{def_K*}
 \end{align}
The result of Sobolev regularity is the following:
\begin{thm}\label{thm2.1}
Let $d=2$ or $3$. We suppose that $B_e\in L^{\infty}_{loc} (\mathbb{R}^{d}, \mathbb{R}^{d(d-1)/2})$ and $V\in W^{1,\infty}_{loc}( \mathbb  R^d \times \mathbb  R^d)$. Then for all $f\in L^{2}_{loc} (\mathbb{R}^{2d})$, such that $ K^* f=g$ with $g\in L^{2}_{loc} (\mathbb{R}^{2d})$, then $ f\in \mathcal{H}^{2}_{loc}(\mathbb{R}^{2d})$\,.
\end{thm}
Before proving Theorem \ref{thm2.1}, we recall the following result~:
\begin{prop}[Proposition A.3 in \cite{zk2}] \label{lem1}
Let $c_j\in L^{\infty,2}_{loc}(\mathbb{R}^{d}\times \mathbb{R}^d),\,\forall j=1,...,d$, where  $L^{\infty,2}_{loc}(\mathbb  R^d \times \mathbb  R^d)=\{u\in L^2_{loc}, \, \forall \varphi\in C^\infty_0 (\mathbb  R^d\times \mathbb  R^d)\text{ such that } \varphi u\in L^\infty_x(L^2_v)\}$, 
such that
\begin{align}\label{cond1}
\partial_{v_j}\,(c_j(x,v))=0\,\text{ in } \mathcal{D}'(\mathbb{R}^{2d})\,,\forall j=1,..,d\,.
\end{align}
Let $P_0$ be the Kolmogorov operator
\begin{align}\label{def:P0}
P_0:= -v\cdot \nabla_x -\Delta_v\,.
\end{align}
 If $h \in L^{2}_{loc}(\mathbb{R}^{2d})$ satisfies
\begin{equation}\label{prob1}
\begin{cases}
P_0h= \sum\limits_{j=1}^{d} c_{j}(x,v)\,\partial_{v_j}\,h_j+\tilde{g}\\
h_j, \tilde{g}\in L^{2}_{loc}(\mathbb{R}^{2d}),\,\forall j=1,...,d \,,
\end{cases}
\end{equation}
 then $ \nabla_v \, h \in L^{2}_{loc} (\mathbb{R}^{2d},\mathbb R^d)\,.$
\end{prop}
We can now give the proof of Theorem \ref{thm2.1}.
\begin{proof}[Proof of Theorem \ref{thm2.1}] The proof is similar to that of Theorem A.2 in \cite{zk2}. In the following, 
we will only focus on the differences appearing  in our case.
To show the Sobolev regularity of the problem $ K^* \, f = g $ with $ f $ and $ g \in L^2_{loc} (\mathbb{R}^{2d}) $, we can reformulate the problem as follows ~:
\begin{align*}
\begin{cases}
P_0\,f= \sum\limits_{j=1}^{d} c_{j}(x,v)\,\partial_{v_j}\,h_j+\tilde{g}\\
h_j=f \in L^{2}_{loc}(\mathbb{R}^{2d})\,,\\
\tilde{g}=g-\frac{v^2}{4}\,f+\frac{d}{2}\,f\in L^{2}_{loc}(\mathbb{R}^{2d}),\,\forall j=1,...,d \,.
\end{cases}
\end{align*}
Here the coefficients $ c_j $ are defined by ~:
\begin{align*}
c_j(x,v)=-(v\wedge\,B_e)_{j} -\partial_{x_j}\,V\in L^{\infty}(\mathbb{R}^{d},L^{2}_{loc}(\mathbb{R}^{d}))\,,\forall j=1,...,d\,.
\end{align*}
We note that the coefficients $ c_j $ verify the condition \eqref{cond1} of the Proposition \ref{lem1} because
$$ \partial_{v_j}(v\wedge B_e)_j=0\, \text{ and  }\, \partial_{v_j}\,\partial_{x_j}\,V=0 \,,\,\forall j=1,..,d\,. $$
\end{proof}

\subsection{Proof of Theorem \ref{prop 2}}~\\
The accretivity of the operator $ K $ is clear.
To show that the operator ${\bf K}$ is maximally accretive, it suffices to show that there exists $ \lambda_{0}> 0 $ such that the operator\break  $ T = K + \lambda_{0} \, Id $ is of dense range  in $ L^2 (\mathbb{R}^{d}\times\mathbb{R}^{d}) $. As in \cite{HeNi}, we take $  \lambda_{0} = \frac{d}{2}+1 \,$. \\
We have to prove that if  $u \in L^2(\mathbb{R}^{d}\times\mathbb{R}^{d})$ satisfies
\begin{align}
\langle u,(K + \lambda_{0} Id)\,w \rangle =0, \quad \forall w \in D(K)\,,
\label{6-z}
\end{align}
then $ u = 0 $.

For this we observe that equation \eqref{6-z} implies that 
$$   K^*\,u=-(\frac{d}{2}+1)u\,\text{ in }\, \mathcal D'(\mathbb{R}^{d}\times\mathbb{R}^{d})\,,$$
where $K^*$ is the operator defined in \eqref{def_K*}.

 Under the assumption that $ B_e\in L^\infty_{loc} (\mathbb{R}^{d},\mathbb{R}^{d(d-1)/2}), V\in W^{1, \infty}_{loc} (\mathbb{R}^d) $ and $ u \in D (K^*) \subset  L^2_{loc} (\mathbb{R}^{d}\times\mathbb{R}^{d}) $,  Theorem \ref{thm2.1} shows that $ u \in \mathcal{H}^2_{loc}(\mathbb{R}^{d}\times\mathbb{R}^{d})$. So we have $\chi (x,v) u \in \mathcal{H}^2(\mathbb{R}^{d}\times\mathbb{R}^{d})$ for any $\chi\in C_0^\infty(\mathbb R^d\times \mathbb R^d)$. 
The rest of the proof is standard. The regularity obtained for $ u $ allows us to justify the integrations by parts and the cut-off argument given in \cite[Proposition 5.5]{HeNi}.
\begin{remark}One can also prove that the operator $K=K_{V,B_e}$ is maximally accretive by using the Kato perturbation theory   by an unbounded operator but this can only be done under the stronger assumption that $B_e\in  L^\infty (\mathbb{R}^{d},\mathbb{R}^{d(d-1)/2})$ and $V\in W^{1, \infty}(\mathbb{R}^d) $. Under this assumption, one can  prove that $\mathcal B=-\nabla_x V\cdot \nabla_v$ is $K_{B_e}$-bounded operator with a $K_{B_e}$-bound which  is strictly smaller than $1$.(cf. \cite[Section A.4]{zkphd}), where $K_{B_e}$ is the same operator with $V=0$ already studied in  \cite{zkphd}.
\end{remark}

\section{Proof of Theorem \ref{hypoelli0}}
\subsection{General strategy}
For technical reasons, it is easier to work with
$$
\check K:= K +\frac d2\,.
$$
It is clear that it is equivalent to prove the maximal estimate for $\check K$.\\

 The proof consists in constructing $ \mathcal{G} $, a graded and stratified algebra of type $2$, and, at any point $ x \in \mathbb{R}^{d} $, an homogeneous element $ \mathcal{F}_{x} $ in the enveloping algebra $ \mathcal{U} _ {2} (\mathcal{G}) $ which satisfies the Rockland condition. We recall that Helffer-Nourrigat's proof is based on maximal estimates which not only hold for the operator but also (and uniformly) for $\pi( \mathcal{F}_{x})$ where $\pi$  is any induced representation of the Lie Algebra. \\
 It remains to find $\pi_x$ such that $\pi_x(\mathcal F_x)=\mathcal K_x +\frac d2$ is a good approximation of ${\bf K}$ in a suitable ball centered at $x$  and 
 to patch together the estimates through a partition of unity. This strategy was used in \cite{HeNi} in the case $B=0$ and in \cite{zkphd,ZK} in the case of the torus $\mathbb T^d$ with $V=0$.\\
 Actually, we first define $\mathcal K_x$ and then look for the Lie Algebra, the operator and the induced representation.
\subsection{Maximal estimate} 
Let us present for simplification the approach when $d=2$ 
(but this restriction is not important). The dependence on $x$ will actually appear through the two parameters $(b,w)\in \mathbb R \times \mathbb R^2$ where 
$$
b=B_e(x) \mbox{ and } w =\nabla V(x)\,.
$$
We then consider the following model:
  \begin{align}
\mathcal K_x +1:=K_{w,b} = v\cdot \nabla_{x} -w.\nabla_v+b(v_1 \partial_{v_{2}} -v_2 \partial_{v_{1}}) -\Delta_v + v^{2}/4\,, \text{ in } \mathbb{R}^{2}\times\mathbb{R}^{2}.
\end{align}
We would like to have uniform estimates with respect to the parameters $b$,  $w$. This is the object of the following proposition.
\begin{prop}\label{propmax}
For any compact interval $I$, there exists a constant $C$ such that for any $b\in I$, any $w\in \mathbb R^2$, any $f\in \mathcal S(\mathbb R^4)$  we have the maximal estimate
\begin{equation}
 |w|^\frac 43 || f ||^2 +  |w|^\frac 23 ||\nabla_v f ||^2+ |w|^\frac 23 ||v f||^2 + \sum_{\vert \alpha \vert+\vert \beta \vert\leq 2} || v^\alpha \partial_v^\beta f ||^2 \leq C \left( ||f ||^2 + ||  K_{w,b} f  ||^2\right) \,.
\end{equation}
\end{prop}
\subsection{Proof of Proposition \ref{propmax}. Step 1}~\\
Following \cite{HeNi}, we consider, after application of 
 the $x\rightarrow \xi$ partial Fourier transform, the family indexed by $w,b,\xi$:
\begin{align}\label{model_1}
\hat K_{w,b,\xi} = i v\cdot \xi -w.\nabla_v+b(v_1 \partial_{v_{2}} -v_2 \partial_{v_{1}}) -\Delta_v + v^{2}/4 , \text{ in } \mathbb{R}^{2}\times\mathbb{R}^{2},
\end{align}
We denote by $\sigma$ the symbol of the operator $\hat K_{w,b,\xi}$ considered as acting in $L^2 (\mathbb  R^2_v)$:
\begin{align*}
\sigma (v,\eta)=i \xi \cdot v -i w\cdot \eta -i b(v_1 \eta_2 -v_2 \eta_1)+\eta^2 +v^2/4.
\end{align*}
We introduce the symplectic map on $\mathbb R^4$ associated with the matrix 
$$A=\begin{pmatrix}
\cos t_1 &-\sin t_1&0&0\\
\sin t_1 &\cos t_1 &0&0\\
0&0&\cos t_2 &-\sin t_2\\
0&0&\sin t_2 &\cos t_2
\end{pmatrix}.$$
Then there exists (see \cite{HeNi} for an explicit definition) an associate unitary metaplectic operator $\tilde{T}$ acting in $L^2 (\mathbb  R^4)$ such that
\begin{equation}\label{conj}
\begin{array}{l}
 \tilde{T}^{-1}\, \hat K_{w,b, \xi}\,\tilde{T}\\
 \quad =  i v\cdot \xi'(t)-w'(t) \cdot \nabla_v+b'_1(t) (v_1 \partial_{v_{2}} -v_2 \partial_{v_{1}})  + i  b'_2(t) (v_1v_2 -\partial_{v_1}\partial_{v_2})  -\Delta_v + v^{2}/4 \,,
 \end{array}
 \end{equation}
with $t=(t_1,t_2)\in \mathbb  R^2$
\begin{align}
 \begin{cases}
w'_k (t)=w_k \cos t_k -\xi_k \sin t_k\\
\xi'_k (t)=w_k \sin t_k +\xi_k \cos t_k\\
b'_1(t)=b\cos (t_2-t_1)\\
b'_2(t)=b\sin (t_2-t_1)
\end{cases}
\label{changes_variables}
\end{align}
for $k=1,2$.\\
 We now choose $t_k$ so that 
$$w'_k (t)=0 \text{ and } \xi'_k (t)=\sqrt{w_k^2+\xi_k^2}:=\rho_k. $$
With this choice  of $t_k$, this leads by \eqref{conj} to the analysis of a maximal estimate for
\begin{align}\label{model_2}
\check K_{\rho,b'}=i  \,v\cdot \rho+ b'_1 (v_1 \partial_{v_2}-v_2\partial_{v_1})+i b'_2(v_1v_2 -\partial_{v_1}\partial_{v_2})-\Delta_v + v^{2}/4 \,,
\end{align}
 which is considered, for fixed $(\rho,b')$, as an operator on $\mathcal S(\mathbb R^2)$.\\   We notice that $|b'|=|b|$. Hence, if $b$ belongs to a compact interval, then $b'$ belongs to a compact set in $\mathbb R^2$.\\
It remains to show that  for a suitable Lie Algebra $\mathcal G$, this is the image by an induced representation $\pi_{\rho}$ of an element $ \mathcal F_{b'} $ satisfying the Rockland condition.

\subsection{Nilpotent techniques for the analysis of $\check K_{\rho,b'}$.}~\\
We construct a graded Lie algebra $ \mathcal G $ of type 2, a subalgebra $ \mathcal{H} $,  for $b'\in \mathbb R^2 $ an element $ \mathcal F_{b'} $ in $ \mathcal U_2 (\mathcal G) $,  and for $\rho \in \mathbb R^2$ a linear form $ \ell_\rho\in \mathcal G^*$ such that $\ell_\rho ([\mathcal H,\mathcal H|)=0 $ and 
 $$ 
  \pi_{\ell_\rho, \mathcal{H}} (\mathcal F_{b'})=\check K_{\rho,b'}\,.
  $$
$ \check K_{\rho,b'}$ can be written as a  $\rho$-independent polynomial of  five differential operators  
  \begin{align}\label{ecriture 17}
\check K_{\rho,b'} &=X_{1,2}-\sum_{k=1}^{2}\, \left( (X_{k,1}^{'})^{2} +\frac{1}{4} (X_{k,1}^{''})^{2} )    \right)\notag\\
&\quad -ib'_1\left( X_{1,1}^{'}X^{''}_{2,1}-X_{2,1}^{'}X^{''}_{1,1}\right)-ib'_2 \left(X^{''}_{1,1}X^{''}_{2,1}+X'_{1,1}X'_{2,1} \right),
\end{align}
which are given by 
  \begin{equation}
 X_{1,1}^{'}=\partial_{v_1} \,,\, X_{1,1}^{''}=iv_{1}\,,\, 
X_{2,1}^{'}=\partial_{v_2} \,,\,  X_{2,1}^{''}=iv_2 \,,\,X_{1,2}=i \,v\cdot\rho\,.
\end{equation}
We now look at the Lie algebra generated by these five operators and their brackets. This leads us to introduce three new elements that verify the following relations  :
\begin{align*}
& X_{2,2}:=[X_{1,1}^{'},X_{1,1}^{''}]=[X_{2,1}^{'},X_{2,1}^{''}]=i\,,\,\\
& X_{1,3}:=[X_{1,2},X_{1,1}^{'}]=-i\rho_1 \,,\, X_{2,3}:=[X_{1,2},X_{2,1}^{'}]=-i\rho_2.
\end{align*}
We also observe that we have the following properties:
\begin{align*}
 &[X_{1,1}^{'},X_{2,1}^{'}]=[ X_{1,1}^{''},X_{2,1}^{''}]=0\,, && \\
 &[ X_{j,1}^{'},X_{k,3}]=[ X_{j,1}^{''},X_{k,3}]=[ X_{k,3},X_{2,2}]=...=0&\,,\, \forall j,k=1,2\,.&
 \end{align*}
 
 We then construct a graded Lie  algebra $ \mathcal{G} $ verifying the same commutator relations. More precisely, $ \mathcal{G} $ is stratified of type $ 2 $, nilpotent of rank $ 3 $, its underlying vector space is $ \mathbb{R}^{8} $, and $ \mathcal{G}_{1} $ is generated by  four elements $ Y_{1,1}^{'}, Y^{'}_{2,1}, Y_{1,1}^{''} $ and $ Y^{'' }_{2,1} $, $ \mathcal{G}_{2} $ is generated by $ Y_{1,2} $ and $ Y_{2,2} $ and $ \mathcal{G}_{3} $ is generated by $ Y_{1,3} $ and $ Y_{2,3} $.
 \begin{align*}
& Y_{2,2}:=[Y_{1,1}^{'},Y_{1,1}^{''}]=[Y_{2,1}^{'},Y_{2,1}^{''}], \,\\
& Y_{1,3}:=[Y_{1,2},Y_{1,1}^{'}] \,,\, Y_{2,3}:=[Y_{1,2},Y_{2,1}^{'}]\,,
\end{align*}
and 
\begin{align*}
 &[Y_{1,1}^{'},Y_{2,1}^{'}]=[ Y_{1,1}^{''},Y_{2,1}^{''}]=0\,, && \\
 &[ Y_{j,1}^{'},Y_{k,3}]=[ Y_{j,1}^{''},Y_{k,3}]=[ Y_{k,3},Y_{2,2}]=...=0&\,,\, \forall j,k=1,2\,.&
 \end{align*} We note that $\mathcal{G}$ is the same graded Lie algebra given in \cite{zk2} and is indeed independent of the parameters $(\rho, b')$.

 We have to check that for a given $\rho=(\rho_1,\rho_2)$ the representation $ \pi $ (with the convention that if $ \diamond = \emptyset $ there is no exponent) defined on its basis by 
 \begin{align}
\pi(Y_{i,j}^{\diamond})= X_{i,j}^{\diamond}\text{ with }i=1,2,\, j=1,2,3\text{ and }\diamond\in \{ \emptyset, \, \prime, \prime\prime\}.
\end{align}
defines an induced representation of the Lie algebra $ \mathcal{G} $. \\
By applying the same steps given in \cite[page 11]{zk2} (following the techniques of \cite{helffer1980hypoellipticite}), 
 we  actually obtain $ \pi = \pi_{\ell, \mathcal{H}} $ with
  \begin{equation}
 \mathcal{H}
  = \mathrm{Vect} (Y_{1,1}^{''}, Y_{2,1}^{''}, Y_{1,2}, Y_{2,2}, Y_{1,3}, Y_{2,3}), 
  \end{equation}
  and $\ell_\rho  \in \mathcal G^*$ defined by $0$ for the elements of the basis of $\mathcal G$  except 
  \begin{equation}
  \ell_\rho (Y_{1,3})=-\rho_1\,, \,\ell_\rho (Y_{2,3})=- \rho_2 \mbox{ and }  \ell_\rho (Y_{2,2})= 1\,.
 \end{equation}
With  \eqref{ecriture 17} in mind, we introduce
\begin{align}
\mathcal F_{b'}=Y_{1,2}&-\sum_{k=1}^{2}\, \left( (Y_{k,1}^{'})^{2} +\frac{1}{4} (Y_{k,1}^{''})^{2} )    \right)\\
&-ib'_1\left( Y_{1,1}^{'}Y^{''}_{2,1}-Y_{2,1}^{'}Y^{''}_{1,1} \right)-ib'_2 \left(Y^{''}_{1,1}Y^{''}_{2,1}+Y'_{1,1}Y'_{2,1} \right)\notag
\end{align}
 and get
\begin{equation}
\pi_{\ell_\rho, \mathcal{H}}(\mathcal F_{ b'}) =\check K_{\rho,b'}\,.
\end{equation}
The verification of the Rockland condition is the same as  in \cite{zk2}. For any non trivial  irreducible representation $\pi$, we consider a $C^\infty$-vector of the representation such that  $\pi (\mathcal F_{b'}) u =0$ and write
$\Re \langle \pi (\mathcal F_{b'}) u \,,\, u\rangle =0$. By using that the operator $ \pi (Y) $ is a formally skew-adjoint operator for all $Y\in \mathcal{G}$ (see \cite[Proposition 2.7]{zk2}), we first get that $\pi(Y) u =0$ for any $Y\in \mathcal G_1$ and by difference $\pi(Y_{1,2}) u=0$.\\
Observing that $\mathcal G$ is the algebra generated by $\mathcal G_1$ and $Y_{12}$, we get $\pi(Y) u =0$ for any $Y\in \mathcal G$, which implies, $\pi$ being non trivial,  that $u=0\,$.\\
 Therefore, according to Helffer-Nourrigat theorem \cite{helffer1978hypoellipticite}  the operator $ \mathcal F_{b'} $ is maximal hypoelliptic and this implies also that $\pi ( \mathcal F_{b'} )$ satisfies a maximal estimate for any induced representation $\pi$ with a constant which is independent of $\pi$.
By applying this argument to  $ \check K_{\rho, b'} = \pi_{\ell, \mathcal{H}} (\mathcal F_{ b'} $), we obtain for any compact $K_0 \subset \mathbb R^2$ the existence of $ C> 0 $ such that, $\forall b'\in K_0$, $\forall \rho \in \mathbb R^2$ and 
$\forall u \in \mathcal S(\mathbb{R}^{2})$,
\begin{align}\label{ing 11}
&\Vert X_{1,2} u\Vert^2 +\sum_{k=1}^{2} \, \left(\Vert (X_{k,1}^{'})^{2}u\Vert^2 +\Vert (X_{k,1}^{''})^{2}u\Vert^2\right) +\sum_{k,\ell=1}^{2} \Vert X^{'}_{k,1}X^{''}_{\ell, 1} u \Vert^2 \notag \\
& \quad  \quad  \qquad  \leq C\, \left( \Vert \check K_{\rho, b'}u\Vert^2+\Vert u \Vert^2 \right) \,.
\end{align}
 Notice here that we first prove the inequality for fixed $b'$ and then show the local uniformity with respect to $b'$.\\
In particular, we have
\begin{equation}\label{ing11a}
\Vert (v\cdot \rho) \,  u\Vert^2 + \sum_{\vert \alpha \vert+\vert \beta \vert\leq 2} || v^\alpha \partial_v^\beta u ||^2  \leq C\, \left( \Vert \check K_{\rho, b'}u\Vert^2+\Vert u \Vert^2 \right) \,.
\end{equation}
Using the treatment of the operator introduced in (V5.52) in \cite{HeNi}, there exists $\check C$, such that, $\forall \rho \in \mathbb R^2$ and 
$\forall u \in \mathcal S(\mathbb{R}^{2})$, we have 
\begin{equation}\label{ing11b}
|\rho|^{\frac 43} || u||^2 \leq \check C ( || (-\Delta + v\cdot \rho) u )||^2)\,.
\end{equation}
Combining \eqref{ing11a} and \eqref{ing11b}, we finally obtain,  for any compact $K_0 \subset \mathbb R^2$ the existence of $ \hat C> 0 $ such that, $\forall b'\in K_0$, $\forall \rho \in \mathbb R^2$ and 
$\forall u \in \mathcal S(\mathbb{R}^{2})$,

\begin{equation}\label{inegmaxhom}
\Vert (v\cdot \rho) \,  u\Vert^2 +  |\rho| ^\frac 43 || u ||^2 + \sum_{\vert \alpha \vert+\vert \beta \vert\leq 2} || v^\alpha \partial_v^\beta u ||^2  \leq \hat C  \left( || \check K_{\rho,b'}   u ||^2+\Vert u\Vert^2\right) \;.
\end{equation}

\subsection{End of the proof of Proposition \ref{propmax}}
Coming back to the initial coordinates, we get for a new constant $C>0$, 
\begin{equation}
||f ||^2 + || \hat K_{w,b, \xi} f  ||^2 \geq \frac 1C \left( |w|^\frac 43 || f ||^2 +  \sum_{\vert \alpha \vert+\vert \beta \vert\leq 2} || v^\alpha \partial_v^\beta f ||^2\right)\;, \forall f \in \mathcal{S}(\mathbb  R^2),
\end{equation}
Note that by complex interpolation, this implies also for $f \in \mathcal{S} (\mathbb  R^2)\;:$
\begin{equation}\label{maxi_est}
||f ||^2 + || \hat K_{w,b, \xi} f  ||^2 \geq \frac 1{\tilde C} \left(  |w|^\frac 23 ||v f ||^2 +  |w|^\frac 23 ||\nabla_v f ||^2+\sum_{\vert \alpha \vert+\vert \beta \vert\leq 2} || v^\alpha \partial_v^\beta f ||^2\right)\,.
\end{equation}
This achieves the proof of the proposition.

\subsection{Proof of Theorem \ref{hypoelli0}}
\subsubsection{Step 1: construction  of the partition of unity.}~\\
We start with the following lemma
\begin{lem}
Let $V$ satisfies \eqref{hyp2/3} for some $\rho_0 \in (\frac 13,1)$ and, for $s>0$, let us define
\begin{equation}
r(x):= <\nabla_x V (x)>^{-s}\,.
\end{equation}
Then there exists $\delta_0 >0$ such that if $|x-y| \leq \, \delta_0  r(x) $ then $$\delta_0 \leq \frac{r(x)}{r(y)} \leq \frac 1{\delta_0}\,.$$
\end{lem}
\begin{proof}
For $c >0$, let $\hat v(x,c):= \sup_{|x-y| \leq c r(x)}  < \nabla V (y) >$.  Using our assumptions, there exists $C_0$, such that, for any $c>0$, we have
$$
\hat  v  (x,c) \leq <\nabla V(x)>  + C_0 \, c \, r(x) \hat v(x,c) ^{1-\rho_0} \,.
$$
Hence
$$
\begin{array}{ll}
\frac{\hat v  (x,c)}{<\nabla V(x)>}  & \leq 1   + C_0 \, c  \, <\nabla V(x)>^{-1 - s}   \hat v(x,c) ^{1-\rho_0} \\
& \leq  1 + C_0 \, c  \, (\frac{\hat v  (x,c)}{<\nabla V(x)>})^{1-\rho_0}   \,.
\end{array}$$
This inequality implies  for all $c >0$  the existence of $C_1(c)$ such that, $\forall x \in \mathbb R^2$,
$$
1 \leq  \frac{\hat v  (x,c)}{<\nabla V(x)>} \leq C_1(c) \,,
$$
and in particular:
$$
 \frac{<\nabla V (y)>}{<\nabla V(x)>} \leq C_1(c) \,, \mbox{ for } |x-y|\leq c  r(x)\,.
$$
To get the reverse inequality, we fix some $c_0$ and consider $c \in (0,c_0)$. We then get for $|x-y|\leq c r(x)|$, 
$$
\begin{array}{ll}
<\nabla V(x) > &\leq <\nabla V(y )>  + C_0\,  c\, | \nabla V(x)|^{-s} \, C_1(c_0)\,  |\nabla V(x)| ^{1-\rho_0}\\
& \leq  <\nabla V(y )>  + C_0 \, c  \, C_1(c_0)\, |\nabla V(x)|  \,.
\end{array}
$$
Choosing $c$ small enough, we obtain 
$$
<\nabla V(x) > \leq (1- C_0 \,c\, C_1(c_0))^{-1}  |\nabla V(y)|  \,.
$$
Hence we have found $c_1 \in (0,c_0)$ and $C_2 >1$ such that $|x-y|\leq c_1 r(x)|$,
\begin{equation}
\frac{1}{C_2}\leq  \frac{<\nabla V (y)>}{<\nabla V(x)>} \leq C_2 \,.
\end{equation}
It is then easy to get the lemma for $\delta_0 \in (0,c_1)$ small enough.
\end{proof}
The parameter $\delta_0$ is now fixed by the lemma and we now consider $\delta \in (0, \delta_0)$ and  $r(x,\delta)=\delta \,  r(x)$. According to Lemma 18.4.4  (and around this lemma in Section 8.4)  in  \cite{Hor2},   one can now introduce,  for any $\delta \in (0,\delta_0]$,   a $\delta$-dependent partition of unity $\phi_j$ in the $x$ variable corresponding to a covering by balls $B(x_j, r(x_j,\delta))$ (with the property of uniform finite
 intersection $N_\delta$) 
  where, for $r>0$  and $\hat x \in \mathbb R^2$, the ball $B(\hat x, r)$ is defined by  
$B(\hat x , r):=\{x\in \mathbb{R}^{2}\,/\, |x-\hat x|<r \}\,. $\\ 
So the support of each $\phi_j$ is contained in  $B(x_j, \delta r(x_j))$ and
 we have
\begin{equation}
\sum_j \phi_j^2 (x) =1\,,
\end{equation}
and 
\begin{equation} 
|\nabla \phi_j (x)| \leq C_\delta  <\nabla V (x_j)>^{s } \leq \hat C_\delta  <\nabla V (x)>^{s }\;.
\end{equation}
This implies (using the finite intersection property) 
\begin{equation} 
\sum_j |\nabla \phi_j (x)|^2 \leq \check C_\delta   <\nabla V (x)>^{2s }\;.
\end{equation}

\subsubsection{Step 2}~\\
The proof is inspired by the Chapter 9 of \cite{HeNi} but different due to the presence of the magnetic field and the absence of the assumption $|\nabla V(x) | \rightarrow +\infty$ at infinity.\\
  We start, for $u\in C_0^\infty(\mathbb  R^2 \times \mathbb  R^2)\,$,  from
\begin{align*}
|| \check K u ||^2 &= \sum_j || \phi_j \check K u ||^2
\\
& = \sum_j || \check K \phi_j u||^2 - \sum_j ||[\check K,\phi_j] u ||^2\\
& = \sum_j || \check K \phi_j u||^2 - \sum_j || (X_0 \phi_j) u ||^2\\
& = \sum_j || \check K \phi_j u||^2 -
\sum_j || (\nabla \phi_j)  \cdot v \, u ||^2\\
& \geq \sum_j || \check K \phi_j u||^2 -\check C_\delta || <\nabla V>^{s} v u ||^2\;.
\end{align*}
For the analysis of $ || \check K \phi_j u||^2 $ let us now write, with $w_j=\nabla V(x_j)$ and $b_j=B_e (x_j)$,
$$
\check K = \check K-K_{w_j,b_j} + K_{w_j,b_j}\,.$$
We verify that, by using the construction of the partition of unity and the assumptions of Theorem \ref{hypoelli0}
$$
\begin{array}{lll}
|| (\check K-K_{w_j, b_j}) \phi_j u ||^2 & \leq 
2 || \phi_j (x) (\nabla V(x) - w_j) \cdot \nabla_v u||^2+2 || \phi_j(x) (B_e(x) - b_j) (v_1\partial_{v_2}-v_2\partial_{v_1})  u||^2\\
&  \leq C \delta^2 \left(  ||\,|\nabla V|^{1-\rho_0 -s} 
\phi_j \nabla_v u ||^2  + ||\,|\nabla V|^{ - s+\gamma_0} 
\phi_j (v_1\partial_{v_2}-v_2\partial_{v_1}) u ||^2\right)\;.
\end{array}
$$
These errors have to be controlled by the main term.

We note that by using the inequality \eqref{maxi_est}, we have
\begin{align*}
|| \phi_j u||^2+ ||K_{w_j,b_j} \phi_j u ||^2 &\geq \frac 1C || \,|\nabla V|^\frac 23  \phi_j u ||^2
 + \frac 1C ||\, |\nabla V |^\frac 13 \phi_j \nabla_v u ||^2 \\
 &\quad  +
 \frac 1C  || \,|\nabla V |^\frac 13 \phi_j v  u ||^2+\frac 1C \sum_{\vert \alpha \vert+\vert \beta \vert\leq 2} || v^\alpha \partial_v^\beta \phi_j  u ||^2\;.
\end{align*}
Finally, we observe that
$$ 
||\check K\phi_ju||^2 \geq \frac 12 || K_{w_j,b_j} \phi_ju||^2 -  ||(\check K- K_{w_j,b_j})\phi_j u||^2
\;.
$$
We now choose
\begin{equation}
\max (\frac 23 -\rho_0, \gamma_0)<s<\frac13\,.
\end{equation}
Summing up over $j$, we have obtained the existence of  a constant $C$ such that for all $u\in C_0^\infty (\mathbb  R^2 \times \mathbb  R^2)$ 
\begin{align}\label{ing_31}
||u||^2+ || \check Ku ||^2 
&\geq \frac 1C ||\, |\nabla V |^\frac 23 \,u \, ||^2
 + \frac 1C  ||\, |\nabla V |^\frac 13 \,\nabla_v u \,||^2
 +  \frac 1C  || \, |\nabla V |^\frac 13\, v u \, ||^2 \notag \\
 &\qquad \qquad +\frac 1C \sum_{\vert \alpha \vert+\vert \beta \vert\leq 2} || v^\alpha \partial_v^\beta u ||^2 \notag
\\ 
&\quad 
 - C \delta^2  || \,  |\nabla V |^{1-s-\rho_0} \, \nabla_v u \, ||^2 -C \delta^2  ||\,|\nabla V|^{ - s+\gamma_0} (v_1\partial_{v_2}-v_2\partial_{v_1}) u ||^2\\\notag
&\,\,- C \left(
\int_{\mathbb  R^d}\int_{\Omega (R)} |\nabla V |^{2s} \, |\nabla_v u |^2\,dxdv+\int_{\mathbb  R^d}\int_{\Omega (R)^c} |\nabla V|^{2s} \, |\nabla_v u |^2\,dxdv\right)\\\notag
& - \hat C  \left(
\int_{\mathbb  R^d}\int_{\Omega (R)} |\nabla V |^{2s} \,  |v u |^2\,dxdv+\int_{\mathbb  R^d}\int_{\Omega (R)^c} |\nabla V |^{2s} \, |v u |^2\,dxdv\right)\\\notag
&\quad -C \delta^2  (||vu||^2 + ||u||^2 + ||\nabla_v u||^2)
\end{align}
where $\Omega (R)$ is defined by
$\Omega (R)=\{x\in \mathbb  R^d \,/\, |\nabla V(x)|<R\}$
and $\Omega (R)^c:=\mathbb  R^d\setminus\Omega (R)$ is the complement of $\Omega (R)$ in $\mathbb  R^d$. \\
Using the definition of $\Omega (R)$ and the assumption that $s<\frac{1}{3}$, we obtain
\begin{align*}
&\int_{\mathbb  R^d}\int_{\Omega (R)} |\nabla V| ^{2s} \, |\nabla_v u |^2\,dxdv \leq R^{2s} \int_{\mathbb  R^d}\int_{\Omega (R)} \, |\nabla_v u |^2\,dxdv,\\
&\int_{\mathbb  R^d}\int_{\Omega (R)^c} |\nabla V| ^{2s} \, |\nabla_v u |^2\,dxdv \leq R^{2s -\frac 23} \int_{\mathbb  R^d}\int_{\Omega (R)^c} |\nabla V| ^{\frac23} \, |\nabla_v u |^2\,dxdv\,,
\end{align*}
and
\begin{align*}
&\int_{\mathbb  R^d}\int_{\Omega (R)} |\nabla V| ^{2s} \, |v u |^2\,dxdv \leq R^{2s} \int_{\mathbb  R^d}\int_{\Omega (R)} \, |v u |^2\,dxdv,\\
&\int_{\mathbb  R^d}\int_{\Omega (R)^c} |\nabla V| ^{2s} \, |v u |^2\,dxdv \leq R^{2s- \frac 23}\, \int_{\mathbb  R^d}\int_{\Omega (R)^c} |\nabla V| ^{\frac23} \, |v u |^2\,dxdv \,.
\end{align*}
Similarly for  $ || \,  |\nabla V| ^{1-s-\rho_0} \, \nabla_v u \, ||^2$, we have, for $1 > s +\rho_0 > \frac 23$
\begin{align*}
&\int_{\mathbb  R^d}\int_{\Omega (R)} |\nabla V| ^{2-2s -2 \rho_0} \, |\nabla_v u |^2\,dxdv \leq R^{2-2 s -2\rho_0}\, \int_{\mathbb  R^d}\int_{\Omega (R)} \, |\nabla_v u |^2\,dxdv,\\
&\int_{\mathbb  R^d}\int_{\Omega (R)^c} |\nabla V| ^{2-2s-2\rho_0} \, |\nabla_v u |^2\,dxdv \leq R^{\frac 43 -2s -2 \rho_0} \, \int_{\mathbb  R^d}\int_{\Omega (R)^c} |\nabla V| ^{\frac23} \, |\nabla_v u |^2\,dxdv\,,
\end{align*}

By combining the previous inequalities in \eqref{ing_31}, we obtain for a possibly new larger $C>0$ and a $(\delta,R)$  dependent constant 
\begin{align*}
||u||^2+ || \check Ku ||^2 &\geq \frac 1C ||\, |\nabla V |^\frac 23 \,u \, ||^2\\
& \quad  + \left(\frac 1C - \check C_\delta  (R^{2s-2/3} + R^{\frac 43 -2 s-2 \rho_0}) \right) \,\left( \,  ||\, |\nabla V|^\frac 13 \,\nabla_v u \,||^2+ |\nabla V |^\frac 13 \, v u \,||^2\right)
\\
& \quad  + (\frac 1C - C \delta^2)  || (v_1\partial_{v_2}-v_2\partial_{v_1}) u ||^2\\
&\quad - C_{\delta,R}  (||\nabla_v u||^2+ ||vu||^2 + ||u||^2) \,.
\end{align*}
We can achieve the proof by observing our conditions on $s$, $\rho_0$,  $\gamma_0$ and by choosing fist $\delta$ small enough and then  $R$ large enough. With this choice of $\delta$ and $R$, we obtain the existence of a constant $C>0$ such that 
\begin{equation*}
\begin{array}{l}
||u||^2+ || \check Ku ||^2 \\ \quad   \geq \frac 1C \left( ||\, |\nabla V|^\frac 23 \,u \, ||^2 + \,  || |\nabla V |^\frac 13 \,\nabla_v u \,||^2+ |\nabla V |^\frac 13 \, v u \,||^2
 + || (v_1\partial_{v_2}-v_2\partial_{v_1}) u ||^2 \right)\\ \qquad \qquad  - C  (||\nabla_v u||^2+ ||vu||^2 + ||u||^2) \,.
\end{array}
\end{equation*}
Combining this inequality with the standard inequality
$$
\Re < \check K u, u>\,  \geq  ||\nabla_v u||^2+ ||vu||^2  \,,\, \forall u \in C_0^\infty (\mathbb R^{2d})\,,
$$
we can  achieve the proof of the theorem, keeping in mind that the maximal estimate for $\check K$ is equivalent to  the maximal estimate for $K$.



\end{document}